\documentclass[a4paper, 11pt]{article}
\usepackage{amsmath}
\usepackage{amsfonts,amsthm,amssymb}
\usepackage{amsfonts}
\usepackage{graphics}
\usepackage{pstricks}
\usepackage{graphicx}
\usepackage{cleveref}

\usepackage{extarrows,chngpage,array,float}     
\usepackage{blkarray} 
\usepackage{amssymb}
\usepackage{latexsym,bm}
\usepackage{cite}

\newtheorem{theorem}{Theorem}[section]
\newtheorem{lemma}[theorem]{Lemma}

\newtheorem{corollary}[theorem]{Corollary}

\newtheorem{question}[theorem]{Question}

\allowdisplaybreaks   

\usepackage{multicol}
\usepackage{makecell}
\usepackage{authblk}

\begin{document}
\title{\textbf{A sharp upper bound for the number of  connected sets in any grid graph}}
\author[a,b]{Hongxia Ma}
\author[a]{Xian'an Jin}
\author[a]{Weiling Yang}
\author[a]{Meiqiao Zhang\thanks{Corresponding Author.\\\indent \hspace{0.18cm} E-mails:  hongxiama516@163.com, xajin@xmu.edu.cn, ywlxmu@xmu.edu.cn, meiqiaozhang95@163.com.}}
	\affil[a]{\small School of Mathematical Sciences, Xiamen University, P. R. China}
	\affil[b]{\small School of Mathematical Sciences, Xinjiang Normal University, P. R. China}
\date{}
\maketitle
\begin{abstract}
A connected set in a graph is a subset of vertices whose  induced subgraph is connected. Although counting the number of connected sets  in a graph is generally a \#P-complete problem, it remains an active area of research. In 2020,   Vince  posed the problem of finding a formula for the number of connected sets in  the $(n\times n)$-grid graph. In this paper, we establish a sharp upper bound for the number of connected sets in any grid graph by using multistep recurrence formulas, which further derives enumeration formulas for the numbers of connected sets in $(3\times n)$- and $(4\times n)$-grid graphs, thus solving a special case of the general problem posed by Vince.
In the process, we also determine the number of connected sets of $K_{m}\times P_{n}$ by employing the transfer matrix method,  where $K_{m}\times P_{n}$ is the Cartesian product of the complete graph of order $m$ and the path of order $n$.
\\
\skip0.2cm
\noindent \textbf{Keywords:} connected sets;  Cartesian product; grid graph; enumeration; recurrence formula.\\
\end{abstract}

\section{Introduction}
\noindent

Throughout this paper, a graph $G$ always means a finite simple graph with the vertex set $V(G)$ and the edge set $E(G)$. The number of connected components of $G$ is denoted by  $c(G)$. 
For any positive integer $n$, let $K_{n}$, $C_{n}$ and $P_{n}$ be the complete
graph, cycle and path of order $n$, respectively.
 The \textit{Cartesian product} of any pair of graphs $G$ and $H$ is the graph $G\times H$ whose vertex set is $V(G)\times V(H)$ and whose edge set is the set of all
pairs $(u_{1},v_{1})(u_{2},v_{2})$ such that either $u_{1}u_{2} \in E(G)$ and $v_{1}=v_{2}$, or $v_{1}v_{2} \in E(H)$ and $u_{1}=u_{2}$.
In particular, the graph $P_{m}\times P_{n}$ is known as the $(m\times n)$-grid graph.\par
For any graph $G$, a \textit{connected set} of $G$ is a subset $C\subseteq V(G)$ such that the subgraph of $G$ induced by $C$, denoted by $G[C]$, is connected. By convention, the empty set is not a connected set.  Let $\mathcal C(G)$ be the set of  all the connected sets of $G$, and let $N(G)=|\mathcal C(G) |$.
 \par

In recent years, extensive research has been devoted to the enumeration of connected sets in arbitrary graphs.
Luo and Xu \cite{Luo} established sharp upper and lower bounds for $N(G)$ of any graph $G$ based on the order of $G$ and certain other parameters, including the chromatic number, stability number, and so on. Moreover, they derived in \cite{Xu} linear-time algorithms for the number and average size of connected sets in a planar 3-tree. 
For results on special classes of graphs,
more research has been conducted in \cite{Vince2022,Ralaivaosaona,KANG2018}. In particular, the problem of counting subtrees in trees has been explored in \cite{KRIK,LI2012,JAMISON1987177,JAMISON1990177,Sills,Stanley,Yan,Haslegrave}, with generalizations to arbitrary graphs discussed in \cite{SUN,CHIN},
while in 2020, Vince  obtained explicit formulas for $N(P_{2}\times P_{n})$ and $N(P_{2}\times C_{n})$ and posed the following problem in  \cite{Wagner2016}.\par
\begin{question}[\cite{Wagner2016}, Question 2]\label{ques}
Find a formula for $N(P_{n}\times P_{n})$, that is, the number of connected sets of the
$(n\times n)$-grid graph.
\end{question}

In this paper, we offer a partial solution to  Question \ref{ques} by developing a general method based on recurrence formulas. We will establish an upper bound for the number of connected sets in any $(m \times n)$-grid graph in Theorem~\ref{main}, which is sharp  when $m=3,4$. Then as further applications, we present enumeration formulas for the numbers of connected sets of $(3\times n)$- and $(4\times n)$-grid graphs in Theorems \ref{main3} and \ref{main4}, respectively.

The outline of this article is as follows.
We first provide an enumeration formula for $N(K_{m}\times P_{n})$ in Theorem~\ref{main1}, with the proof given in Section 2 by using the transfer matrix method.
\begin{theorem}\label{main1}
For any positive integers $m$ and $n$, the total number of connected sets in $K_{m}\times P_{n}$  is
$$N(K_{m}\times P_{n})=\sum\limits_{k=1}^{n}(n-k+1)\cdot(\mathbf{u}\cdot T^{k-1}\mathbf{1} ),$$
where
\text{\( \mathbf{u} = \begin{bmatrix}\binom{m} {1} & \binom{m} {2}  & \cdots & \binom{m} {m}\end{bmatrix} \)}, the transfer matrix $T=(t_{ij})_{i,j\in \{1,2,...,m\}}$, and
\[
t_{ij} =
\begin{cases}
\binom{m} {j}-\binom{m-i} {j}, & \text{when } j\leq m-i, \\
 \binom{m} {j}, & \text{when } j\geq  m-i+1.
\end{cases}
\]
\end{theorem}

Next, in Section 3, we develop Lemmas \ref{lem6} and \ref{lem1} that provide a method based on multistep recurrence formulas to determine a lower bound of $|\mathcal C(K_{m}\times P_{n} )\setminus \mathcal C(P_{m}\times P_{n})|$, which is sharp when $m=3,4$. Then using the equation that
$$N(P_{m}\times P_{n})=N(K_{m}\times P_{n})-|\mathcal C(K_{m}\times P_{n} )\setminus \mathcal C(P_{m}\times P_{n})|,$$
we are able to establish an upper bound of $N(P_{m}\times P_{n})$ in Theorem~\ref{main} that
is tight when $m=3,4$.
Subsequently, the detailed derivation of enumeration formulas for the numbers of connected sets of  $(3\times n)$- and $(4\times n)$-grid graphs will be  given in Section 4, which shows promise of calculating for the case $m\ge 5$, but also reveals the extreme complexity of this type of enumeration problems.

\section{Calculating $N(K_{m}\times P_{n})$}
In this section, we prove Theorem~\ref{main1} and  list clearly the  case for $m=3,4$ as examples.

We first give vertex labels of $K_{m}\times P_{n}$ for any positive integers $m$ and $n$, which will be used throughout the remaining of this paper.
Assume that $V(K_{m}\times P_{n})=\{ v_{i,j}: i\in \{1,2,...,m\},  j \in \{1,2,...,n\}\}$, where for any $i_1,i_2\in \{1,2,...,m\}$ and $j_1,j_2\in \{1,2,...,n\}$, $v_{i_1,j_1}v_{i_2,j_1}\in E(K_{m}\times P_{n})$, and $v_{i_1,j_1}v_{i_1,j_2}\in E(K_{m}\times P_{n})$ if and only if $|j_1-j_2|=1$. For each $j=1,2,\dots, n$, we call the set $I_j=\{v_{i,j}:i \in \{1,2,...,m\}\}$  the \textit{$j$-th column} of $K_{m}\times P_{n}$, the induced graph of which is clearly isomorphic to $K_m$. For any $S_{j}\subseteq I_{j}$ with $j\ge 2$,
let $S_{j-1}$ be the set of vertices in $I_{j-1}$ that are adjacent to the vertices in $S_{j}$. Then $|S_{j}|=|S_{j-1}|$ clearly holds.

Now we give the proof of Theorem~\ref{main1}.

\noindent\textit{Proof of Theorem~\ref{main1}.}
For any  integer $1\le k\le n$, let $\mathcal C_{m,k}$ be the set of connected sets of $K_{m}\times P_{k}$ that contain at least one vertex in each column, and let $N_{m,k}=|\mathcal C_{m,k}|$. Then
\begin{eqnarray}\label{equ2.1}
N(K_{m}\times P_{n})=\sum\limits_{k=1}^{n}(n-k+1) \cdot N_{m,k}.
\end{eqnarray}
Thus it suffices to consider the values of $N_{m,k}$ for all integers $1\le k\le n$ in the following.

For any $i=1,2,\dots,m$, let $f_{m,k}^{i}$ denote the number of  connected sets  in $K_{m}\times P_{k}$ that contain at least one vertex in each column and in particular, a  fixed
set of  $i$ vertices in the $k$-th column. Then $f_{m,1}^{i}=1$ for all $i=1,2,\dots,m$ and
\begin{eqnarray}\label{equ2.2}
N_{m,k}=\sum \limits_{i=1}^{m}\binom{m} {i}f_{m,k}^{i}.
\end{eqnarray}

Recall that $I_{k}$ is the set of vertices in the $k$-th column of  $K_{m}\times P_{k}$. Then
for  any $C\in \mathcal C_{m,k}$ with $C\cap I_{k}=S_{k}$, it is clear to see that $C\cap S_{k-1}\neq \emptyset$ and $C\setminus S_{k} \in \mathcal C_{m,k-1}$.
Thus we have a recurrence formula that for any $k\geq 2$ and $i\in\{1,2,\dots,k\}$,
\begin{eqnarray}\label{equ2.3}
f_{m,k}^{i}&= &N_{m,k-1}-\sum \limits_{j=1}^{m-i}\binom{m-i} {j}f_{m,k-1}^{j}\nonumber\\
&=&\sum \limits_{j=1}^{m}\binom{m} {j}f_{m,k-1}^{j}-\sum \limits_{j=1}^{m-i}\binom{m-i} {j}f_{m,k-1}^{j}\nonumber\\
&=&\sum \limits_{j=1}^{m-i}\left(\binom{m} {j}-\binom{m-i} {j}\right)f_{m,k-1}^{j}+\sum \limits_{j=m-i+1}^{m}\binom{m} {j}f_{m,k-1}^{j}.
\end{eqnarray}

  For any $k\ge 2$, let \text{\( \mathbf{v}_{k} = \begin{bmatrix}f_{m,k}^{1} &f_{m,k}^{2} & \cdots & f_{m,k}^{m}\end{bmatrix} \)}.
Then by (\ref{equ2.3}), we have
  \begin{eqnarray}\label{equ2.4}
\mathbf{v_{k}^{\top}}=T \mathbf{v_{k-1}^{\top}}=T^{k-1}\mathbf{v_{1}^{\top}}=T^{k-1} \mathbf{1},
  \end{eqnarray}
 where $\mathbf{v_{1}^{\top}}=\mathbf{1}$ is the vector whose entries are all 1's.

Further, let \text{\( \mathbf{u} = \begin{bmatrix}\binom{m} {1} & \binom{m} {2}  & \cdots & \binom{m} {m}\end{bmatrix} \)}.
Then by (\ref{equ2.2}) and (\ref{equ2.4}), we have
\begin{eqnarray}\label{equ2.5}
N_{m,k}=\mathbf{u}\cdot \mathbf{v}_{k}^{\top}=\mathbf{u}\cdot T^{k-1}\mathbf{1}.
\end{eqnarray}

The result follows from (\ref{equ2.1}) and (\ref{equ2.5}).
\hfill$\Box$

By Theorem \ref{main1}, our enumeration formulas for $N(K_{3}\times P_{n})$ and $N(K_{4}\times P_{n})$ are as follows.
\begin{corollary}\label{coro1}
The total number of connected sets of $K_{3}\times P_{n}$ is
$$N(K_{3}\times P_{n})=\sum\limits_{k=1}^{n} (n-k+1)\cdot N_{3,k},$$
where \[
N_{3,k}= \mathbf{u_{1}}\cdot T_{1}^{k-1}\mathbf{1}, \quad    \mathbf{u_{1}} = \begin{bmatrix}\binom{3} {1} & \binom{3} {2}  & \binom{3} {3}\end{bmatrix},   \quad    \text{and} \quad
\begin{array}{c}
T_{1}=\left[ \begin{matrix}
1 &2  &1 \\
2 & 3 & 1 \\
3 & 3 & 1
\end{matrix} \right]
\end{array}.
\]
\end{corollary}
\par
\noindent\textbf{Remark:}
With a little effort, we can get $N_{3,1}=N(K_{3})=7$,  $N_{3,2}= 37$, $ N_{3,3}= 205$, and $ N_{3,4}= 1129$. \par

\begin{corollary}\label{coro2}
The total number of connected sets of $K_{4}\times P_{n}$ is
$$N(K_{4}\times P_{n})=\sum\limits_{k=1}^{n}(n-k+1)\cdot N_{4,k},$$
where
\[
N_{4,k}= \mathbf{u_{2}}\cdot T_{2}^{k-1}\mathbf{1}, \quad    \mathbf{u_{2}} =\begin{bmatrix}\binom{4} {1} & \binom{4} {2}  & \binom{4} {3} & \binom{4} {4}\end{bmatrix},   \quad    \text{and} \quad
\begin{array}{c}
T_{2}=\left[ \begin{matrix}
1 &3  &3 &1 \\
2 & 5  &4& 1 \\
3 & 6  &4& 1 \\
4 & 6  &4& 1
\end{matrix} \right]
\end{array}.
\]

\end{corollary}

\noindent\textbf{Remark:}
With a little bit more effort, we can get $N_{4,1}=N(K_{4})=15$, $N_{4,2}= 175$, $ N_{4,3}= 2129$, and $ N_{4,4}= 25793$. \par

\section{An upper bound of $N(P_{m}\times P_{n})$}
In this section, we establish a sharp upper bound of $N(P_{m}\times P_{n})$ by developing a general scheme to classify all the connected sets in $\mathcal{C}(K_{m}\times P_{n})\setminus \mathcal{C}(P_{m}\times P_{n})$.

For any positive integers $m$ and $n$, since $P_{m}\times P_{n}$ is a subgraph of $K_{m}\times P_{n}$, we assume that $V(P_{m}\times P_{n})=V(K_{m}\times P_{n})$, where for any $i_1,i_2\in\{1,2,\dots,m\}$ and $j\in\{1,2,\dots,n\}$,  $v_{i_1,j}v_{i_2,j}\in E(P_{m}\times P_{n})$ if and only if $|i_1-i_2|=1$.
Then it is clear to see the identity mapping from $ V(P_{m}\times P_{n})$ to $V(K_{m}\times P_{n})$ induces a homomorphism from $P_{m}\times P_{n} $ to $K_{m}\times P_{n}$. Therefore,
 $\mathcal C(P_{m}\times P_{n} )\subseteq  \mathcal C(K_{m}\times P_{n} )$. In other words, we can obtain  $\mathcal C(P_{m}\times P_{n} )$ by removing some elements from $\mathcal C(K_{m}\times P_{n} )$, and any such element must contain a pair of vertices that are adjacent in $K_{m}\times P_{n}$ but not adjacent in $P_{m}\times P_{n}$ (the reverse may  not be true). In fact, it is indeed such pairs of vertices that make this element a connected set of $K_{m}\times P_{n}$ but not of $P_{m}\times P_{n}$.

 In the following, we shall provide a method to determine a lower bound of $|\mathcal C(K_{m}\times P_{n} )\setminus \mathcal C(P_{m}\times P_{n} )|$ by recurrence formulas, by which along with Theorem~\ref{main1}, an upper bound of $N(P_{m}\times P_{n})$ can be obtained.

Recall that $\mathcal C_{m,n}$ is the subset of $\mathcal C(K_{m}\times P_{n})$ that every element in $\mathcal C_{m,n}$ contains at least one vertex in each column of $K_{m}\times P_{n}$. Let $\mathcal C_{m,n}'$ be the subset of
 $\mathcal C(K_{m}\times P_{n} )\setminus \mathcal C(P_{m}\times P_{n} )$ that every element in $\mathcal C_{m,n}'$ contains at least one vertex in each column of $K_{m}\times P_{n}$. Then $\mathcal C_{m,n}'\subseteq \mathcal C_{m,n}$.
For any nonempty set $X\subseteq V(K_{m}\times P_{n})$, let
 \begin{eqnarray}
N(K_{m}\times P_{n};X)&=&|\{C\in \mathcal{C}_{m,n}: C\cap I_i=X\cap I_i~\text{whenever}~X\cap I_i\neq\emptyset\}|,\nonumber\\
N'(K_{m}\times P_{n};X)&=&|\{C\in \mathcal{C}_{m,n}': C\cap I_i=X\cap I_i~\text{whenever}~X\cap I_i\neq\emptyset\}|.\nonumber
 \end{eqnarray}

Then
\begin{equation}
N(K_{m}\times P_{n};X)=f_{m,n}^{|X|}~\text{if}~X\subseteq I_n,\nonumber
\end{equation}
 where $f_{m,n}^{|X|}$ is defined in the proof of Theorem~\ref{main1} and can be obtained by the recurrence formula (\ref{equ2.3}).
 Moreover, the following lemma clearly holds.
 \begin{lemma}\label{lem6}
For any positive integers $m$ and $n$,
\begin{equation}\label{equ3.6}
|\mathcal C_{m,n}'|=\sum_{\emptyset\neq S_n\subseteq I_n}N'(K_{m}\times P_{n};S_n),
\end{equation}
where
\begin{equation}\label{equ3.16}
 N'(K_{m}\times P_{n};S_n)=\sum_{\emptyset\neq X\subseteq I_{n-1} \atop X\cap S_{n-1}\neq \emptyset
 }N'(K_{m}\times P_{n};S_n\cup X)~\text{when}~n\ge 2.
 \end{equation}
 \end{lemma}

%
By (\ref{equ3.16}), it is routine to calculate the value of $N'(K_{m}\times P_{n};S_n)$ when $n=1,2$. Then for any $n\ge 3$, we provide a recurrence method in the next lemma to give a lower bound of $N'(K_{m}\times P_{n};S_n)$.

 \begin{lemma}\label{lem1}
 For any $n\ge 3$ and $C \in \mathcal C_{m,n}'$  with $C\cap I_{n}= S_{n}$ and $C\cap I_{n-1}= T_{n-1}$,
 \begin{enumerate}
 \item[(1)] if some component of $P_{m}\times P_{n}[S_{n}]$ is also a component of $P_{m}\times P_{n}[S_{n}\cup T_{n-1}]$,
 then
 $N'(K_{m}\times P_{n}; S_{n}\cup T_{n-1})=N(K_{m}\times P_{n-1}; T_{n-1} )$;
 \item[(2)] if each component of $P_{m}\times P_{n}[S_{n}]$ is not a component of $P_{m}\times P_{n}[S_{n}\cup T_{n-1}]$ and $c(P_{m}\times P_{n}[S_{n}\cup T_{n-1}])=c(P_{m}\times P_{n-1}[ T_{n-1}])$, then  $N'(K_{m}\times P_{n}; S_{n}\cup T_{n-1} )=N'(K_{m}\times P_{n-1}; T_{n-1} )$;
 \item[(3)] for the remaining case that each component of $P_{m}\times P_{n}[S_{n}]$ is not a component of $P_{m}\times P_{n}[S_{n}\cup T_{n-1}]$ and
 $c(P_{m}\times P_{n}[S_{n}\cup T_{n-1}])< c(P_{m}\times P_{n-1}[ T_{n-1}])$, we have
 \begin{equation}\label{equ3.11}
 N'(K_{m}\times P_{n}; S_{n}\cup T_{n-1} )= \sum_{\emptyset\neq X\subseteq I_{n-2} \atop X\cap T_{n-2}\neq \emptyset
 }N'(K_{m}\times P_{n}; S_{n}\cup T_{n-1} \cup X).
 \end{equation}
Moreover,
 let $\mathcal{X}$ be the set of subsets $X$ of $I_{n-2}$ with $X\cap T_{n-2}\neq\emptyset$ and satisfying one of the following conditions:
  \begin{enumerate}
  \item[(i)] some component of $P_{m}\times P_{n}[S_{n}\cup T_{n-1}]$ is also a component of $P_{m}\times P_{n}[S_{n}\cup T_{n-1}\cup X]$;
  \item[(ii)] each component of $P_{m}\times P_{n}[S_{n}\cup T_{n-1}]$ is not a component of $P_{m}\times P_{n}[S_{n}\cup T_{n-1}\cup X]$ and $c(P_{m}\times P_{n}[S_{n}\cup T_{n-1}\cup X])=c(P_{m}\times P_{n-2}[X])$;
   \item[(iii)]
each component of $P_{m}\times P_{n}[S_{n}\cup T_{n-1}]$ is not a component of $P_{m}\times P_{n}[S_{n}\cup T_{n-1}\cup X]$, $c(P_{m}\times P_{n}[S_{n}\cup T_{n-1}\cup X])<c(P_{m}\times P_{n-2}[X])$, and for any two integers $i_1$ and $i_2$, $l_{i_1}=l_{i_2}$ if and only if $l_{i_3}=l_{i_1}$ holds for all integers $i_3$ with $i_1\le i_3\le i_2$,
where the components of $P_{m}\times P_{n-2}[X]$  are $R_{1} ,R_{2},\dots R_r$, $k_s< k_t$ for
 all vertices $v_{k_s,n-2}$ in $R_s$ and  $v_{k_t,n-2}$ in $R_t$ whenever $s<t$, and
 $l_{i}$ (or $\overline{l}_{i}$) is the smallest (or largest)  integer $q$ such that $v_{q,n-2}$ is in the component of $P_{m}\times P_{n}[S_{n}\cup T_{n-1}\cup X]$ containing $R_i$ for all $i=1,\dots,r$.
 \end{enumerate}

 Then
 \begin{equation}\label{equ3.12}
 N'(K_{m}\times P_{n}; S_{n}\cup T_{n-1} )\ge \sum_{X\in \mathcal{X}}N'(K_{m}\times P_{n}; S_{n}\cup T_{n-1} \cup X),
  \end{equation}
and
 \begin{itemize}
 \item[(3.1)] $N'(K_{m}\times P_{n}; S_{n}\cup T_{n-1}\cup X)=N(K_{m}\times P_{n-2};X)$ if $X$ satisfies   condition (i);
 \item[(3.2)] $N'(K_{m}\times P_{n}; S_{n}\cup T_{n-1}\cup X )=N'(K_{m}\times P_{n-2}; X)$ if $X$ satisfies  condition (ii);
  \item [(3.3)]
$N'(K_{m}\times P_{n}; S_{n}\cup T_{n-1}\cup X)=N'(K_{m}\times P_{n-1}; X\cup \{v_{j,n-1}: l_{i}\le j\le \overline{l}_{i}, i=1,\dots,r\})$ if $X$ satisfies  condition (iii).
 \end{itemize}
  \end{enumerate}
 \end{lemma}

For an illustration of Lemma~\ref{lem1}, see some examples shown in Figures~\ref{fig6} and~\ref{fig6'}, where $S_n=C\cap I_n$,  $T_{n-1}=C\cap I_{n-1}$ and $X=C\cap I_{n-2}$.

\begin{figure}[h]
\centering
\scalebox{0.9}[0.9]{\includegraphics{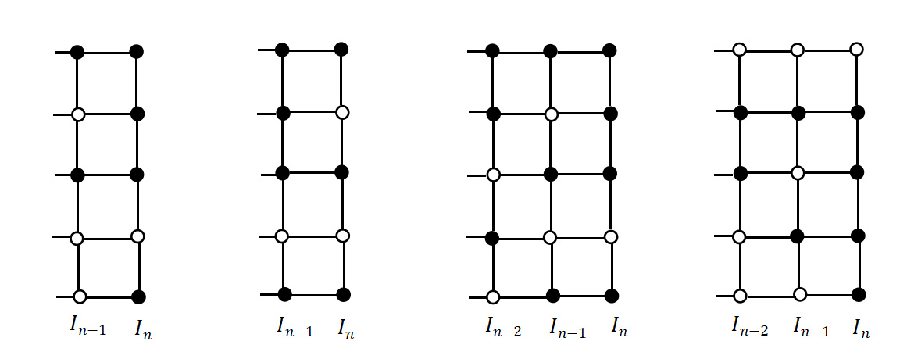}}\\
\hspace{-0.4cm} (a) \hspace{2.4 cm} (b) \hspace{2.9 cm} (c) \hspace{3 cm} (d)
\caption{Examples (a), (b), (c), (d) for Lemma~\ref{lem1} (1), (2), (3) (i), (3) (ii), respectively, where $C$ is the set of solid vertices}
\label{fig6}
\end{figure}

\begin{figure}[h]
\centering
\scalebox{0.85}[0.85]{\includegraphics{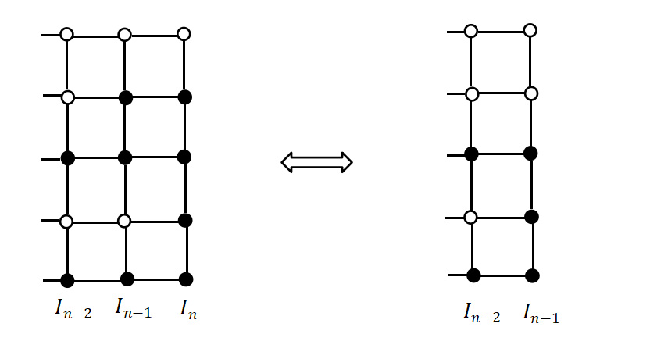}}
\caption{An example for the illustration of Lemma~\ref{lem1} (3.3), where the graph on the left is an example for Lemma~\ref{lem1} (3) (iii) with $C$  the set of solid vertices}
\label{fig6'}
\end{figure}

Now we prove Lemma~\ref{lem1}.

\noindent\textit{Proof of Lemma~\ref{lem1}.}
 For (1),
 since some component of $P_{m}\times P_{n}[S_{n}]$ is also a component of $P_{m}\times P_{n}[S_{n}\cup T_{n-1}]$,
 both $S_n$ and $S_{n}\cup T_{n-1}$ are not connected sets in $P_m\times P_n$. Then it can be easily verified  that any $C$ with $C\cap I_{n}= S_{n}$ and $C\cap I_{n-1}= T_{n-1}$ is a member in  $\mathcal{C}_{m,n}'$ if and only if $C\setminus S_n$ is a member of $\mathcal{C}_{m,n-1}$ with $(C\setminus S_n)\cap I_{n-1}= T_{n-1}$. (1) holds.

 For (2),
 every component of $P_{m}\times P_{n}[S_{n}\cup T_{n-1}]$ is the union (via vertex identifications) of one component of $P_{m}\times P_{n-1}[T_{n-1}]$ and some (maybe 0) components of $P_{m}\times P_{n}[S_{n}]$.
 Then it is easy to see that any $C$ with $C\cap I_{n}= S_{n}$ and $C\cap I_{n-1}= T_{n-1}$ is a member in  $\mathcal{C}_{m,n}'$ if and only if $C\setminus S_n$ is a member of $\mathcal{C}_{m,n-1}'$ with $(C\setminus S_n)\cap I_{n-1}= T_{n-1}$. (2) holds.

 For (3), note that in this case, there must exist a component of $P_{m}\times P_{n}[S_{n}\cup T_{n-1}]$ that is the union of one component of $P_{m}\times P_{n}[S_{n}]$ and at least two components of $P_{m}\times P_{n-1}[T_{n-1}]$. Recall that $T_{n-2}$ is the set of vertices in $I_{n-2}$ that are adjacent to vertices in $T_{n-1}$. Then all the statements except for (3.3) can be similarly verified.

 For (3.3), for any triple $(S_n,T_{n-1}, X)$ satisfying condition (iii),  there must further exist a component of $P_{m}\times P_{n}[S_{n}\cup T_{n-1}\cup X]$ that is the union of one component of $P_{m}\times P_{n}[S_{n}\cup T_{n-1}]$ and at least two  components of $P_{m}\times P_{n-2}[X]$, which implies that $l_s=l_t$ for some distinct $R_s$ and $R_t$.
Moreover, all the components $R_i$ of $P_{m}\times P_{n-2}[X]$ with the same value of $l_i$ belong to the same component in $P_{m}\times P_{n}[S_{n}\cup T_{n-1}\cup X]$, and vice versa.
Then by the condition that for any two integers $i_1$ and $i_2$, $l_{i_1}=l_{i_2}$ if and only if $l_{i_3}=l_{i_1}$ holds for all $i_1\le i_3\le i_2$, we have that all vertices $v_{p,n-2}\in X$ with $l_i\le p\le \overline{l}_i$ belong to the same component of
$P_{m}\times P_{n}[S_{n}\cup T_{n-1}\cup X]$.
Note that $l_i=l_{i'}$ if and only if $\overline{l}_i=\overline{l}_{i'}$, and further, all  intervals $(l_i,\overline{l}_i)$ are either identical or pairwise disjoint.
Then all  vertices $v_{p,n-2}\in X$ with $l_i\le p\le \overline{l}_i$ belong to the same component in
 $P_{m}\times P_{n-1}[X\cup \{v_{j,n-1}: l_{i}\le j\le \overline{l}_{i}, i=1,\dots,r\}]$.
Hence $N'(K_{m}\times P_{n}; S_{n}\cup T_{n-1}\cup X)=N'(K_{m}\times P_{n-1}; X\cup \{v_{j,n-1}: l_{i}\le j\le \overline{l}_{i}, i=1,\dots,r\})$.
The example given in Figure~\ref{fig6'} also explains that there is a natural bijection between such two sets.

 The proof is complete.
\hfill$\Box$

\noindent\textbf{Remark.} Note that for the lower bound given in (\ref{equ3.12}), the only excluded case has the characteristics that each component of $P_{m}\times P_{n}[S_{n}\cup T_{n-1}]$ is not a component of $P_{m}\times P_{n}[S_{n}\cup T_{n-1}\cup X]$, $c(P_{m}\times P_{n}[S_{n}\cup T_{n-1}\cup X])<c(P_{m}\times P_{n-2}[X])$, and there exist  integers $i_1, i_2, i_3$ with $i_1\le i_3\le i_2$ such that $l_{i_1}=l_{i_2}$ but $l_{i_1}\neq l_{i_3}$.
Unfortunately, sometimes we are not able to find a recurrence formula to calculate $N'(K_m\times P_n;S_n\cup T_{n-1}\cup X)$ for this case. See an example in Figure~\ref{fig1}.
For this kind of $C$ with $S_n=C\cap I_n$,  $T_{n-1}=C\cap I_{n-1}$ and $X=C\cap I_{n-2}$, it seems extremely difficult to find an $X'\subseteq I_{n-1}$ such that $N'(K_m\times P_n;S_n\cup T_{n-1}\cup X)=N'(K_m\times P_{n-1};X'\cup X)$.
However, such issues do not arise when $m=3,4$,
which we shall see  in the next section.

\begin{figure}[h]
\centering
\scalebox{0.85}[0.85]{\includegraphics{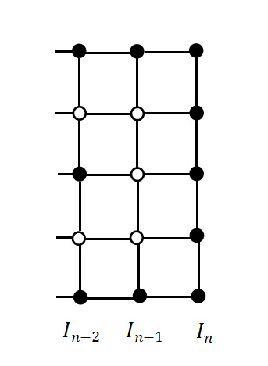}}
\caption{An example of the case that Lemma~\ref{lem1} can not deal with, where $C$ is the set of solid vertices}
\label{fig1}
\end{figure}

Now it is clear to see that Lemmas~\ref{lem6} and~\ref{lem1} together provide a lower bound of $|\mathcal{C}_{m,n}'|$ for any positive integer $n$, denoted by $N'_{m,n}$, where $N'_{m,n}$ can be calculated by multistep recurrence formulas.
We shall show in the next section that this lower bound is tight, i.e.,  $N'_{m,n}=|\mathcal{C}_{m,n}'|$,
when $m=3,4$.
Then by Theorem~\ref{main1}, we obtain a sharp upper bound for $ N(P_{m}\times P_{n})$.
\begin{theorem}\label{main}
For any positive integers $m$ and $n$,	
$$ N(P_{m}\times P_{n})\le \sum\limits_{k=1}^{n}(N_{m,k}-N'_{m,k})\cdot(n-k+1),$$
where the equality holds when $m=3,4$.
\end{theorem}

The proof for the equality case will be given in Lemmas~\ref{coro3} and~\ref{coro4} by explaining that the equality of (\ref{equ3.12}) holds for all corresponding cases when $m=3,4$.

\section{Calculating $N(P_3\times P_n)$ and $N(P_4\times P_n)$}
In this section, we present enumeration formulas for $N(P_3\times P_n)$ and $N(P_4\times P_n)$ in Subsections 4.1 and 4.2, respectively.

In the following, we simply write $N(K_{m}\times P_{n};  \{x_{1},x_{2},...,x_{j}\})$ and $N'(K_{m}\times P_{n}; \mathcal \{x_{1},x_{2},...,x_{j}\})$ as
$N(K_{m}\times P_{n}; x_{1},x_{2},...,x_{j})$ and $N'(K_{m}\times P_{n}; x_{1},x_{2},...,x_{j})$, respectively.

\subsection{Calculating $N(P_3\times P_n)$}
In this subsection, we shall apply Corollary~\ref{coro1}, Lemmas~\ref{lem6} and~\ref{lem1} to calculate $N(P_3\times P_n)$.

For any positive integer $n$ and $i=1,2,3$, assume that $v_{4,n}=v_{1,n}$, and let
$$a_{n}^{i}=N'(K_{3}\times P_{n}; v_{i,n}),~
b_{n}^{i}= N'(K_{3}\times P_{n};  v_{i,n}, v_{i+1,n}),~
c_{n}=N'(K_{3}\times P_{n}; I_n).$$
Referring to Figure~\ref{fig2}, it is clear to see that
 $a_{n}^{1}=a_{n}^{3}$ and $b_{n}^{1}=b_{n}^{2}. $
\begin{figure}[h]
\centering
\scalebox{0.65}[0.65]{\includegraphics{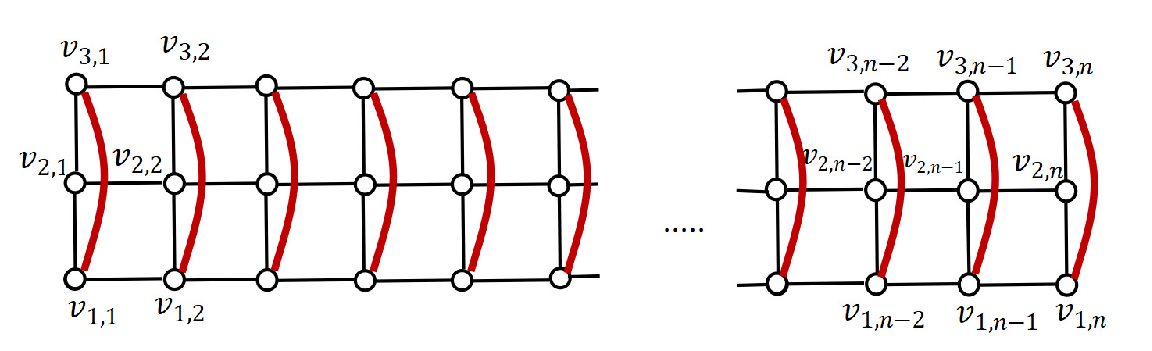}}
\caption{The graph $K_{3}\times P_{n}$}
\label{fig2}
\end{figure}

\begin{lemma}\label{coro3}
For any $n\ge 2$ and $C \in \mathcal C_{3,n}'$ with $C\cap I_{n}= S_{n}$ and $C\cap I_{n-1}= T_{n-1}$,\\
(1) if $c(P_{3}\times P_{n}[S_{n}\cup T_{n-1}])> c(P_{3}\times P_{n-1}[ T_{n-1}])$, then $N'(K_{3}\times P_{n}; S_{n}\cup T_{n-1} )=N(K_{3}\times P_{n-1}; T_{n-1} )$;\\
(2) if $c(P_{3}\times P_{n}[S_{n}\cup T_{n-1}])=c(P_{3}\times P_{n-1}[ T_{n-1}])$, then $N'(K_{3}\times P_{n}; S_{n}\cup T_{n-1} )=N'(K_{3}\times P_{n-1}; T_{n-1} )$;\\
(3) otherwise, $c(P_{3}\times P_{n}[S_{n}\cup T_{n-1}])<c(P_{3}\times P_{n-1}[ T_{n-1}])$, which implies that $S_{n}=  I_{n}$ and $T_{n-1}=\{v_{1,n-1}, v_{3,n-1}\}$.
  Let $d_n=N'(K_{3}\times P_{n};I_n\cup\{v_{1,n-1},v_{3,n-1}\})$. Then $d_2=0$, and  $d_n=2a_{n-2}^{1}+2b_{n-2}^{1}+d_{n-1}+c_{n-2}$ when $n\ge 3$.
\end{lemma}

\begin{proof}
Since $1\le c(X)\le 2$ holds for any $\emptyset\neq X\subseteq I_k$ with $k=1,2,\dots,n$, (1) and (2) can be easily verified according to Lemma~\ref{lem1} (1) and (2).

For (3), it is clear that $S_{n}=  I_{n}$ and $T_{n-1}=\{v_{1,n-1}, v_{3,n-1}\}$.
Then it remains to show the recurrence formula for  $d_n$.

Obviously, $d_2=0$. Assume $n\ge 3$ in the following. Then by (\ref{equ3.11}),
\begin{eqnarray}
d_{n}&=&N'(K_{3}\times P_{n};I_n\cup\{v_{1,n-1},v_{3,n-1}\})\nonumber\\
&=&N'(K_{3}\times P_{n};I_n\cup\{v_{1,n-1},v_{3,n-1},v_{1,n-2}\})+N'(K_{3}\times P_{n};I_n\cup\{v_{1,n-1},v_{3,n-1}, v_{3,n-2}\})\nonumber\\
&&+N'(K_{3}\times P_{n};I_n\cup\{v_{1,n-1},v_{3,n-1}, v_{1,n-2},v_{2,n-2}\})\nonumber\\
&&+N'(K_{3}\times P_{n};I_n\cup\{v_{1,n-1},v_{3,n-1},v_{1,n-2}, v_{3,n-2}\})\nonumber\\
&&+N'(K_{3}\times P_{n};I_n\cup\{v_{1,n-1},v_{3,n-1},v_{2,n-2},v_{3,n-2}\})\nonumber\\
&&+N'(K_{3}\times P_{n};I_n\cup\{v_{1,n-1},v_{3,n-1}\}\cup I_{n-2}).\nonumber
\end{eqnarray}

Note that by Lemma~\ref{lem1} (3.3),
\begin{eqnarray}
&&N'(K_{3}\times P_{n}; I_{n}\cup \{v_{1,n-1},v_{3,n-1},v_{1,n-2},v_{3,n-2}\})\nonumber\\
&=&N'(K_{3}\times P_{n-1}; I_{n-1}\cup \{v_{1,n-2},v_{3,n-2}\})\nonumber\\
&=&d_{n-1}, \nonumber
\end{eqnarray}
with an example for explanation in Figure~\ref{fig3}.

\begin{figure}[h]
\centering
\scalebox{0.85}[0.85]{\includegraphics{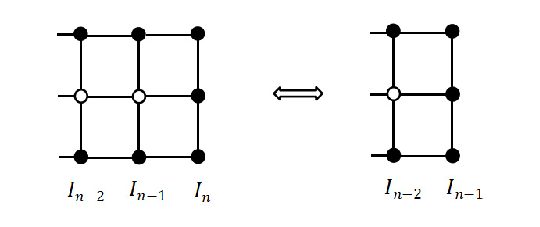}}
\caption{An example}
\label{fig3}
\end{figure}

Then by Lemma~\ref{lem1} (3.2),
\begin{eqnarray}
d_n&=&N'(K_{3}\times P_{n-2}; v_{1,n-2})+N'(K_{3}\times P_{n-2};v_{3,n-2})\nonumber\\
&&+N'(K_{3}\times P_{n-2};v_{1,n-2},v_{2,n-2})
+d_{n-1}\nonumber\\
&&+N'(K_{3}\times P_{n-2}; v_{2,n-2},v_{3,n-2})+N'(K_{3}\times P_{n-2};I_{n-2})\nonumber\\
&=&a_{n-2}^{1}+a_{n-2}^{3}+b_{n-2}^{1}+d_{n-1}+b_{n-2}^{2}+c_{n-2}\nonumber\\
&=&2a_{n-2}^{1}+2b_{n-2}^{1}+d_{n-1}+c_{n-2},\nonumber
\end{eqnarray}
which also implies that (\ref{equ3.12}) is tight when $m=3$.

The proof is complete.
\end{proof}

Since the lower bound $N'_{m,n}$ is sharp when $m=3$, we are able to determine the exact value of $|\mathcal C_{3,n}'|$ by recurrence formulas.
\begin{lemma}\label{lem2}
For any positive integer $n$,
$$|\mathcal C_{3,n}'|=2 a_{n}^{1}+a_{n}^{2}+2 b_{n}^{1}+b_{n}^{3}+ c_{n},$$
where
$a_{1}^{1}=a_{1}^{2}=b_{1}^{1}=c_{1}=0$, $b_{1}^{3}=1$, and when $n\ge 2$,
 \begin{align}\left\{\begin{aligned}
a_{n}^{1}&=a_{n-1}^{1}+b_{n-1}^{1}+b_{n-1}^{3}+c_{n-1},\\
a_{n}^{2}&= a_{n-1}^{2}+2b_{n-1}^{1}+c_{n-1},\\
b_{n}^{1}&=a_{n-1}^{1}+a_{n-1}^{2}+2b_{n-1}^{1}+b_{n-1}^{3}+c_{n-1},\\
b_{n}^{3}&=2f_{3,n-1}^{1}+2f_{3,n-1}^{2}+b_{n-1}^{3}+c_{n-1},\\
c_{n}&=2a_{n-1}^{1}+a_{n-1}^{2}+2b_{n-1}^{1}+d_{n}+c_{n-1}.\nonumber
 \end{aligned}\right.
 \end{align}

 \end{lemma}

\begin{proof}  By (\ref{equ3.6}), we only need to verify the formulas for $a_{n}^{1}, a_{n}^{2}, b_{n}^{1}, b_{n}^{3}, c_{n}$ along with their respective initial values. The case when $n=1$ is obvious. Assume $n\ge 2$ in the following.
 \par

By (\ref{equ3.16}) and Lemma~\ref{coro3} (2), we have
\begin{eqnarray}
a_{n}^{1}&=& N'(K_{3}\times P_{n}; v_{1,n})\nonumber\\
&=&N'(K_{3}\times P_{n};  v_{1,n},v_{1,n-1})+ N'(K_{3}\times P_{n};  v_{1,n},v_{1,n-1}, v_{2,n-1})\nonumber\\
&&+ N'(K_{3}\times P_{n}; v_{1,n},v_{1,n-1}, v_{3,n-1})+ N'(K_{3}\times P_{n}; \{v_{1,n}\}\cup I_{n-1})\nonumber\\
&=&N'(K_{3}\times P_{n-1};  v_{1,n-1})+ N'(K_{3}\times P_{n-1};  v_{1,n-1}, v_{2,n-1})\nonumber\\
&&+ N'(K_{3}\times P_{n-1}; v_{1,n-1}, v_{3,n-1})
+ N'(K_{3}\times P_{n-1}; I_{n-1})\nonumber\\
&=&a_{n-1}^{1}+b_{n-1}^{1}+b_{n-1}^{3}+c_{n-1}.\nonumber
\end{eqnarray}

Similarly, we have
\begin{eqnarray}
a_{n}^{2}&= &N'(K_{3}\times P_{n}; v_{2,n})\nonumber\\
&=&N'(K_{3}\times P_{n-1};  v_{2,n-1})+ N'(K_{3}\times P_{n-1}; v_{2,n-1}, v_{1,n-1})\nonumber\\
&&+ N'(K_{3}\times P_{n-1};  v_{2,n-1}, v_{3,n-1})
+ N'(K_{3}\times P_{n-1}; I_{n-1})\nonumber\\
&=&a_{n-1}^{2}+b_{n-1}^{1}+b_{n-1}^{2}+c_{n-1}\nonumber\\
&=&a_{n-1}^{2}+2b_{n-1}^{1}+c_{n-1},\nonumber
\end{eqnarray}
and
\begin{eqnarray}
b_{n}^{1}&= &N'(K_{3}\times P_{n};  v_{1,n}, v_{2,n})\nonumber\\
&=&N'(K_{3}\times P_{n-1};  v_{1,n-1})
+N'(K_{3}\times P_{n-1}; v_{2,n-1})
+ N'(K_{3}\times P_{n-1}; v_{1,n-1},  v_{2,n-1})\nonumber\\
&&+ N'(K_{3}\times P_{n-1};  v_{1,n-1}, v_{3,n-1})+ N'(K_{3}\times P_{n-1};  v_{2,n-1}, v_{3,n-1})+ N'(K_{3}\times P_{n-1}; I_{n-1})\nonumber\\
&=&a_{n-1}^{1}+a_{n-1}^{2}+b_{n-1}^{1}+b_{n-1}^{3}+b_{n-1}^{2}+c_{n-1}\nonumber\\
&=&a_{n-1}^{1}+a_{n-1}^{2}+2b_{n-1}^{1}+b_{n-1}^{3}+c_{n-1}.\nonumber
\end{eqnarray}

By (\ref{equ3.16}) and Lemma~\ref{coro3} (1), (2), we have
\begin{eqnarray}
b_{n}^{3}&=& N'(K_{3}\times P_{n}; v_{1,n}, v_{3,n} )\nonumber\\
&=&N'(K_{3}\times P_{n}; v_{1,n}, v_{3,n} , v_{1,n-1})+ N'(K_{3}\times P_{n}; v_{1,n}, v_{3,n} , v_{3,n-1})\nonumber\\
&&+N'(K_{3}\times P_{n}; v_{1,n}, v_{3,n} , v_{1,n-1}, v_{2,n-1})
+ N'(K_{3}\times P_{n}; v_{1,n}, v_{3,n}, v_{1,n-1}, v_{3,n-1})\nonumber\\
&&+N' (K_{3}\times P_{n};  v_{1,n}, v_{3,n}, v_{2,n-1}, v_{3,n-1})+N'(K_{3}\times P_{n};  \{v_{1,n}, v_{3,n}\} \cup I_{n-1})\nonumber\\
&=&N (K_{3}\times P_{n-1}; v_{1,n-1})+ N(K_{3}\times P_{n-1}; v_{3,n-1})
+N(K_{3}\times P_{n-1}; v_{1,n-1}, v_{2,n-1})\nonumber\\
&&+ N'(K_{3}\times P_{n-1}; v_{1,n-1}, v_{3,n-1})+N (K_{3}\times P_{n-1};  v_{2,n-1}, v_{3,n-1})+N'(K_{3}\times P_{n-1}; I_{n-1})\nonumber\\
&=&2f_{3,n-1}^{1}+2f_{3,n-1}^{2}+b_{n-1}^{3}+c_{n-1}.\nonumber
\end{eqnarray}

Finally,
by (\ref{equ3.16}) and Lemma~\ref{coro3} (2), (3),  we have
\begin{eqnarray}
c_{n}&=& N'(K_{3}\times P_{n}; I_n)\nonumber\\
&=&N'(K_{3}\times P_{n-1};v_{1,n-1})+ N'(K_{3}\times P_{n-1}; v_{2,n-1})+N'(K_{3}\times P_{n-1}; v_{3,n-1})\nonumber\\
&&+N '(K_{3}\times P_{n-1}; v_{1,n-1}, v_{2,n-1})+
d_n+ N'(K_{3}\times P_{n-1}; v_{2,n-1}, v_{3,n-1})\nonumber\\
&&+N'(K_{3}\times P_{n-1}; I_{n-1})\nonumber\\
&=&a_{n-1}^{1}+a_{n-1}^{2}+a_{n-1}^{3}+b_{n-1}^{1}+d_{n}+b_{n-1}^{2}+c_{n-1}\nonumber\\
&=&2a_{n-1}^{1}+a_{n-1}^{2}+2b_{n-1}^{1}+d_{n}+c_{n-1}.\nonumber
\end{eqnarray}

The proof is complete.
\end{proof}

By Corollary~\ref{coro1} and Lemma~\ref{lem2}, we obtain an enumeration formula for $ N(P_{3}\times P_{n})$, which is in alignment with Theorem~\ref{main}.
\begin{theorem}\label{main3}
For any positive integer $n$, the total number of  connected sets of $P_{3}\times P_{n}$ is
$$ N(P_{3}\times P_{n})=\sum\limits_{k=1}^{n}(N_{3,k}-|\mathcal C_{3,k}'|)\cdot(n-k+1).$$
\end{theorem}
\noindent\textbf{Remark.} See some detailed calculation results below, where $N_{m,n}^{*}$ denote the number of connected sets of $P_{m}\times P_{n}$ that contain at least one vertex in each column.\par
(1) When $n=1$, $|\mathcal C_{3,1}'|=1$, $N_{3,1}^{*}=N(P_{3})=6=N_{3,1}-|\mathcal C_{3,1}'|$.\par
(2) For $n=2$, $n=3$ and $n=4$, \par
~~~~(i) $a_{2}^{1}=1$,  $a_{2}^{2}=0$, $b_{2}^{1}=1$, $b_{2}^{3}=5$, $c_{2}=0$, $|\mathcal C_{3,2}'|=9$, $N_{3,2}^{*}=28$;\par
~~~~(ii) $a_{3}^{1}=7$,  $a_{3}^{2}=2$, $b_{3}^{1}=8$, $b_{3}^{3}=25$, $c_{3}=4$, $|\mathcal C_{3,3}'|=61$, $N_{3,3}^{*}=144$;\par
~~~~(iii) $a_{4}^{1}=44$,  $a_{4}^{2}=22$, $b_{4}^{1}=54$, $b_{4}^{3}=141$, $c_{4}=40$, $|\mathcal C_{3,4}'|=399$, $N_{3,4}^{*}=730$.\par
Using Mathematica, we calculate $N_{3,3}^{*}$ and $N_{3,4}^{*}$, and the results are consistent with those obtained from the  formulas.\par

\subsection{Calculating $N(P_{4}\times P_{n})$}

In this subsection, we derive a formula for $N(P_{4}\times P_{n})$ by applying a method analogous to that used for $N(P_{3}\times P_{n})$.

For any positive integer $n$, let
\begin{itemize}
  \item $a_{n}^{i}=N'(K_{4}\times P_{n};  v_{i,n})$ for $i=1,2,3,4$;
  \item $b_{n}^{i}=N'(K_{4}\times P_{n}; v_{i,n}, v_{i+1,n})$ for $i=1,2,3$, \\
      $b_{n}^{4}=N'(K_{4}\times P_{n};  v_{1,n}, v_{3,n})$,\\
      $b_{n}^{5}=N'(K_{4}\times  P_{n}; v_{2,n},  v_{4,n})$,\\ $b_{n}^{6}=N'(K_{4}\times  P_{n};   v_{1,n}, v_{4,n})$;
  \item $c_{n}^{i}=N'(K_{4}\times  P_{n}; v_{i,n}, v_{i+1,n}, v_{i+2,n})$ for $i=1,2$,\\
  $c_{n}^{3}=N'(K_{4}\times  P_{n};  v_{1,n}, v_{2,n}, v_{4,n})$,\\
  $c_{n}^{4}=N'(K_{4}\times  P_{n}; v_{1,n}, v_{3,n}, v_{4,n})$;
  \item $g_{n}=N'(K_{4}\times  P_{n}; I_n)$.
\end{itemize}
Referring to Figure~\ref{fig4}, it is clear to see that
$$ a_{n}^{1}=a_{n}^{4},~a_{n}^{2}=a_{n}^{3},~ b_{n}^{1}=b_{n}^{3},~ b_{n}^{4}=b_{n}^{5}, ~c_{n}^{1}=c_{n}^{2},~ c_{n}^{3}=c_{n}^{4}.
$$

Then we have the following lemma.

\begin{figure}[h]
\centering
\scalebox{0.65}[0.65]{\includegraphics{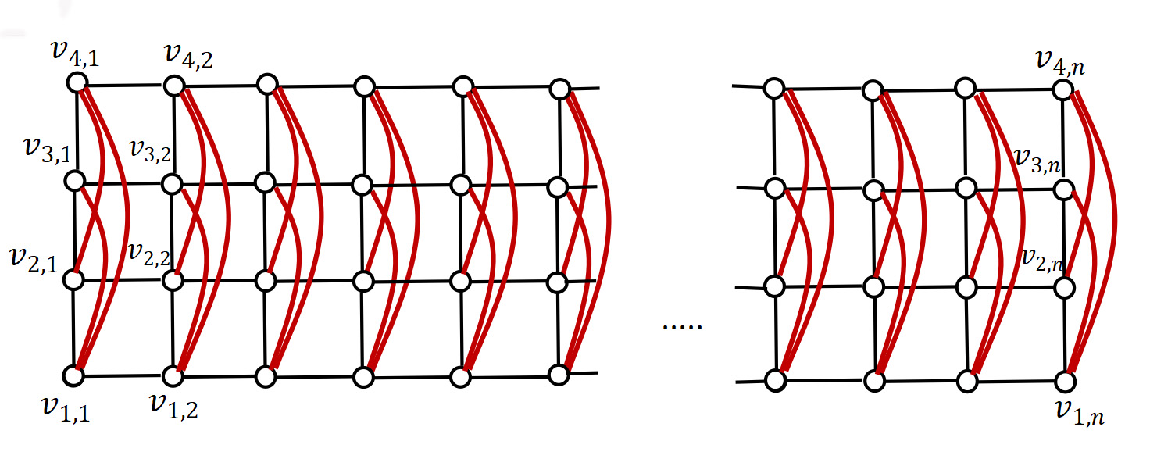}}
\caption{The graph $K_{4}\times P_{n}$ }
\label{fig4}
\end{figure}

\begin{lemma}\label{coro4}
For any $n\ge 2$ and $ C \in \mathcal C_{4,n}'$ with $C\cap I_{n}= S_{n}$ and $C\cap I_{n-1}= T_{n-1}$,\\
(1) if $c(P_{4}\times P_{n}[S_{n}\cup T_{n-1}])> c(P_{4}\times P_{n-1}[ T_{n-1}])$, then $N'(K_{4}\times P_{n}; S_{n}\cup T_{n-1} )=N(K_{4}\times P_{n-1}; T_{n-1} )$;\\
(2) if
$c(P_{4}\times P_{n}[S_{n}\cup T_{n-1}])= c(P_{4}\times P_{n-1}[ T_{n-1}])$,  then $N'(K_{4}\times P_{n}; S_{n}\cup T_{n-1} )=N'(K_{4}\times P_{n-1}; T_{n-1} )$;\\
(3) otherwise,
$c(P_{4}\times P_{n}[S_{n}\cup T_{n-1}])< c(P_{4}\times P_{n-1}[T_{n-1}])$, and all such pairs $(S_n,T_{n-1})$ are listed in Table 1. Let $r_n,t_n,k_n$ be as defined in Table 1. Then
 $r_2=t_2=k_2=0,$ and when $n\ge 3$,
 \begin{align}\left\{\begin{aligned}
r_{n}&=a_{n-2}^{1}+a_{n-2}^{3}+2b_{n-2}^{1}+b_{n-2}^{2}+b_{n-2}^{6}+r_{n-1}+2c_{n-2}^{1}+c_{n-2}^{3}+t_{n-1}+g_{n-2},\\
t_{n}&=2a_{n-2}^{1}+a_{n-2}^{3}+2b_{n-2}^{1}+b_{n-2}^{2}+b_{n-2}^{5}+r_{n-1}+k_{n-1}+2c_{n-2}^{1}+2 t_{n-1}+g_{n-2},\\
k_{n}&=2a_{n-2}^{1}+2b_{n-2}^{1}+2b_{n-2}^{4}+k_{n-1}+2c_{n-2}^{1}+2 t_{n-1}+g_{n-2}.\nonumber
 \end{aligned}\right.
 \end{align}

	\begin{table}[!h]\centering
	\begin{tabular}{|c|c|c|}
	\hline
$S_n$ & $T_{n-1}$ & $N'(K_{4}\times P_{n}; S_{n}\cup T_{n-1} )$ \\
	\hline
$\{v_{1,n},v_{2,n},v_{3,n}\}$ & $\{v_{1,n-1},v_{3,n-1}\}$ & $r_n$\\
	\hline
$\{v_{1,n},v_{2,n},v_{3,n}\}$ & $\{v_{1,n-1},v_{3,n-1},v_{4,n-1}\}$ & $t_n$ \\
	\hline
$\{v_{2,n},v_{3,n},v_{4,n}\}$ & $\{v_{2,n-1},v_{4,n-1}\}$ & $r_n$ \\
	\hline
$\{v_{2,n},v_{3,n},v_{4,n}\}$ & $\{v_{1,n-1},v_{2,n-1},v_{4,n-1}\}$ & $t_n$ \\
	\hline
$I_n$ & $\{v_{1,n-1},v_{3,n-1}\}$ & $r_n$ \\
	\hline
$I_n$ & $\{v_{1,n-1},v_{4,n-1}\}$ & $k_n$ \\
	\hline
$I_n$ & $\{v_{2,n-1},v_{4,n-1}\}$ & $r_n$ \\	
\hline
$I_n$ & $\{v_{1,n-1},v_{2,n-1},v_{4,n-1}\}$ & $t_n$ \\
	\hline
$I_n$ & $\{v_{1,n-1},v_{3,n-1},v_{4,n-1}\}$ & $t_n$ \\
	 \hline
	 \end{tabular}
	 \caption{All $(S_{n},T_{n-1})$ pairs for Lemma~\ref{coro4} (3)}
		\end{table}
\end{lemma}

\begin{proof}
Since $1\le c(X)\le 2$ holds for any $\emptyset\neq X\subseteq I_k$ with $k=1,2,\dots,n$, (1) and (2) can be easily verified according to Lemma~\ref{lem1} (1) and (2).

For (3), it is clear to see that all the possible pairs of $(S_n,T_{n-1})$  are listed in Table 1, and as shown in  Table 1, some pairs of $(S_n,T_{n-1})$ have the same value of $N'(K_{4}\times P_{n}; S_{n}\cup T_{n-1} )$.
Then it remains to show the recurrence formulas of $r_n,t_n,k_n$.

Clearly, $r_{2}=t_{2}=k_{2}=0$.
Assume $n\ge 3$ in the following.

By (\ref{equ3.11}), Lemma~\ref{lem1} (3.2) and (3.3),
\begin{eqnarray}
r_{n}&=&N'(K_{4}\times  P_{n}; v_{1,n}, v_{2,n}, v_{3,n},v_{1,n-1}, v_{3,n-1})\nonumber\\
&=&a_{n-2}^{1}+a_{n-2}^{3}+b_{n-2}^{1}+r_{n-1}+b_{n-2}^{6}+b_{n-2}^{2}+b_{n-2}^3+c_{n-2}^{1}+c_{n-2}^{3}+t_{n-1}\nonumber\\
&&+c_{n-2}^{2}+g_{n-2}\nonumber\\
&=&a_{n-2}^{1}+a_{n-2}^{3}+2b_{n-2}^{1}+b_{n-2}^{2}+b_{n-2}^{6}+r_{n-1}+2c_{n-2}^{1}+c_{n-2}^{3}+t_{n-1}+g_{n-2},\nonumber
\end{eqnarray}
where the second equality holds due to Lemma~\ref{lem1} (3.3) that
 \begin{align}\left\{\begin{aligned}
&N'(K_{4}\times  P_{n}; v_{1,n}, v_{2,n}, v_{3,n}, v_{1,n-1}, v_{3,n-1}, v_{1,n-2}, v_{3,n-2})=r_{n-1},\\
&N'(K_{4}\times  P_{n}; v_{1,n}, v_{2,n}, v_{3,n}, v_{1,n-1}, v_{3,n-1},  v_{1,n-2}, v_{3,n-2},v_{4,n-2})=t_{n-1}.\nonumber
 \end{aligned}\right.
 \end{align}

Similarly,
\begin{eqnarray}
t_{n}&=&N'(K_{4}\times  P_{n}; v_{1,n}, v_{2,n}, v_{3,n},v_{1,n-1}, v_{3,n-1}, v_{4,n-1})\nonumber\\
&=&a_{n-2}^{1}+a_{n-2}^{3}+a_{n-2}^{4}+b_{n-2}^{1}+b_{n-2}^{2}+b_{n-2}^{3}+b_{n-2}^{5}+r_{n-1}+k_{n-1}+c_{n-2}^{1}\nonumber\\
&&+c_{n-2}^{2}+2 t_{n-1}+g_{n-2}\nonumber\\
&=&2a_{n-2}^{1}+a_{n-2}^{3}+2b_{n-2}^{1}+b_{n-2}^{2}+b_{n-2}^{5}+r_{n-1}+k_{n-1}+2c_{n-2}^{1}+2 t_{n-1}+g_{n-2},\nonumber
\end{eqnarray}
where the second equality holds as
 \begin{align}\left\{\begin{aligned}
&N'(K_{4}\times  P_{n}; v_{1,n}, v_{2,n}, v_{3,n}, v_{1,n-1}, v_{3,n-1},v_{4,n-1}, v_{1,n-2}, v_{3,n-2})=r_{n-1},\\
&N'(K_{4}\times  P_{n}; v_{1,n}, v_{2,n}, v_{3,n}, v_{1,n-1}, v_{3,n-1}, v_{4,n-1}, v_{1,n-2}, v_{4,n-2})=k_{n-1},\\
&N'(K_{4}\times  P_{n}; v_{1,n}, v_{2,n}, v_{3,n}, v_{1,n-1}, v_{3,n-1}, v_{4,n-1}, v_{1,n-2}, v_{2,n-2},v_{4,n-2})=t_{n-1},\\
&N'(K_{4}\times  P_{n}; v_{1,n}, v_{2,n}, v_{3,n}, v_{1,n-1}, v_{3,n-1},v_{4,n-1}, v_{1,n-2}, v_{3,n-2},v_{4,n-2})=t_{n-1}.\nonumber
 \end{aligned}\right.
 \end{align}

Finally,
\begin{eqnarray}
k_{n}&=&N'(K_{4}\times  P_{n}; I_{n}\cup\{ v_{1,n-1}, v_{4,n-1}\})\nonumber\\
&=&a_{n-2}^{1}+a_{n-2}^{4}+b_{n-2}^{1}+b_{n-2}^{3}+b_{n-2}^{4}+b_{n-2}^{5}+k_{n-1}\nonumber\\
&&+c_{n-2}^{1}+c_{n-2}^{2}+2 t_{n-1}+g_{n-2}\nonumber\\
&=&2a_{n-2}^{1}+2b_{n-2}^{1}+2b_{n-2}^{4}+k_{n-1}+2c_{n-2}^{1}+2 t_{n-1}+g_{n-2},\nonumber
\end{eqnarray}
where the second equality holds as
 \begin{align}\left\{\begin{aligned}
&N'(K_{4}\times  P_{n}; I_n\cup\{ v_{1,n-1}, v_{4,n-1}, v_{1,n-2}, v_{4,n-2}\})=k_{n-1},\\
&N'(K_{4}\times  P_{n}; I_n\cup\{ v_{1,n-1}, v_{4,n-1}, v_{1,n-2}, v_{2,n-2},v_{4,n-2}\})=t_{n-1},\\
&N'(K_{4}\times  P_{n}; I_n\cup\{ v_{1,n-1}, v_{4,n-1}, v_{1,n-2}, v_{3,n-2}, v_{4,n-2}\})=t_{n-1}.\nonumber
 \end{aligned}\right.
 \end{align}

The proof is complete and it is clear that (\ref{equ3.12}) is tight when $m=4$.
\end{proof}

Since the lower bound $N'_{m,n}$ is also sharp when $m=4$, we are able to determine the exact value of $|\mathcal C_{4,n}'|$ by recurrence formulas.
\begin{lemma}\label{lem3}
For any positive integer $n$,
$$|\mathcal C_{4,n}'|=2 a_{n}^{1}+2 a_{n}^{2}+2 b_{n}^{1}+b_{n}^{2}+2 b_{n}^{4}+b_{n}^{6}+2 c_{n}^{1}+2 c_{n}^{3}+g_{n},$$
where
 $a_{1}^{1}=a_{1}^{2}=b_{1}^{1}=b_{1}^{2}=c_{1}^{1}=g_{1}=0,~b_{1}^{4}=b_{1}^{6}=c_{1}^{3}=1$, and when $n\ge 2$,
 \begin{align}\left\{\begin{aligned}
a_{n}^{1}&=a_{n-1}^{1}+b_{n-1}^{1}+b_{n-1}^{4}+b_{n-1}^{6}+c_{n-1}^{1}+2c_{n-1}^{3}+g_{n-1},\\
a_{n}^{2}&=a_{n-1}^{2}+b_{n-1}^{1}+b_{n-1}^{2}+b_{n-1}^{4}+2c_{n-1}^{1}+c_{n-1}^{3}+g_{n-1},\\
b_{n}^{1}&=a_{n-1}^{1}+a_{n-1}^{2}+b_{n-1}^{1}+b_{n-1}^{2}+2b_{n-1}^{4}+b_{n-1}^{6}+2c_{n-1}^{1}+2c_{n-1}^{3}+g_{n-1},\\
b_{n}^{2}&=2a_{n-1}^{2}+2b_{n-1}^{1}+b_{n-1}^{2}+2b_{n-1}^{4}+2c_{n-1}^{1}+2c_{n-1}^{3}+g_{n-1},\\
b_{n}^{4}&=2f_{4,n-1}^{1}+4f_{4,n-1}^{2}+b_{n-1}^{4}+2f_{4,n-1}^{3}+c_{n-1}^{1}+c_{n-1}^{3}+g_{n-1},\\
b_{n}^{6}&=2f_{4,n-1}^{1}+4f_{4,n-1}^{2}+b_{n-1}^{6}+2f_{4,n-1}^{3}+2c_{n-1}^{3}+g_{n-1},\\
c_{n}^{1}&=a_{n-1}^{1}+2a_{n-1}^{2}+2b_{n-1}^{1}+b_{n-1}^{2}+b_{n-1}^{4}+b_{n-1}^{6}+r_{n}+2c_{n-1}^{1}+c_{n-1}^{3}+t_{n}+g_{n-1},\\
c_{n}^{3}&=3f_{4,n-1}^{1}+4f_{4,n-1}^{2}+b_{n-1}^{4}+b_{n-1}^{6}+f_{4,n-1}^{3}+c_{n-1}^{1}+2c_{n-1}^{3}+g_{n-1},\\
g_{n}&=2a_{n-1}^{1}+2a_{n-1}^{2}+2b_{n-1}^{1}+b_{n-1}^{2}+2  r_{n}+k_{n}+2c_{n-1}^{1}+2 t_{n}+g_{n-1}.\nonumber
 \end{aligned}\right.
 \end{align}
\end{lemma}

\begin{proof}
By (\ref{equ3.6}), we need only to verify the formulas for $a_{n}^{1},~ a_{n}^{2}, ~b_{n}^{1},~ b_{n}^{2}, ~b_{n}^{4}, ~b_{n}^{6}, ~c_{n}^{1}, ~c_{n}^{3}, ~g_{n}$ along with their respective initial values. The case when $n=1$ is obvious. Now assume $n\ge 2$.

By (\ref{equ3.16}) and Lemma~\ref{coro4} (2), we have
\begin{eqnarray}
a_{n}^{1}&=&N'(K_{4}\times  P_{n};  v_{1,n})\nonumber\\
&=&N'(K_{4}\times  P_{n-1};  v_{1,n-1})+ N'(K_{4}\times  P_{n-1}; v_{1,n-1}, v_{2,n-1})+N'(K_{4}\times  P_{n-1};v_{1,n-1}, v_{3,n-1})\nonumber\\
&&+ N'(K_{4}\times  P_{n-1};  v_{1,n-1}, v_{4,n-1})+N'(K_{4}\times  P_{n-1};  v_{1,n-1}, v_{2,n-1} , v_{3,n-1})\nonumber\\
&&+ N'(K_{4}\times  P_{n-1};  v_{1,n-1}, v_{2,n-1} , v_{4,n-1})+ N'(K_{4}\times  P_{n-1}; v_{1,n-1}, v_{3,n-1} , v_{4,n-1})\nonumber\\
&&+ N'(K_{4}\times  P_{n-1}; I_{n-1})\nonumber\\
&=&a_{n-1}^{1}+b_{n-1}^{1}+b_{n-1}^{4}+b_{n-1}^{6}+c_{n-1}^{1}+c_{n-1}^{3}+c_{n-1}^{4}+g_{n-1}\nonumber\\
&=&a_{n-1}^{1}+b_{n-1}^{1}+b_{n-1}^{4}+b_{n-1}^{6}+c_{n-1}^{1}+2c_{n-1}^{3}+g_{n-1},\nonumber
\end{eqnarray}
and the recurrence formulas for $a_{n}^{2}, ~b_{n}^{1}, ~b_{n}^{2}$ can be derived in a similar manner.

By (\ref{equ3.16}) and Lemma~\ref{coro4} (1), (2),
\begin{eqnarray}
b_{n}^{4}&=&N'(K_{4}\times  P_{n}; v_{1,n}, v_{3,n})\nonumber\\
&=&N(K_{4}\times  P_{n-1}; v_{1,n-1})+N(K_{4}\times  P_{n-1};  v_{3,n-1})\nonumber\\
&&+N(K_{4}\times  P_{n-1}; v_{1,n-1},v_{2,n-1})+N'(K_{4}\times  P_{n-1};  v_{1,n-1},v_{3,n-1})\nonumber\\
&&+N(K_{4}\times  P_{n-1}; v_{1,n-1},v_{4,n-1})+N(K_{4}\times  P_{n-1}; v_{2,n-1},v_{3,n-1})\nonumber\\
&&+N(K_{4}\times  P_{n-1};  v_{3,n-1},v_{4,n-1})+N'(K_{4}\times  P_{n-1};  v_{1,n-1},v_{2,n-1},v_{3,n-1})\nonumber\\
&&+N(K_{4}\times  P_{n-1};  v_{1,n-1},v_{2,n-1},v_{4,n-1})+N'(K_{4}\times  P_{n-1};  v_{1,n-1},v_{3,n-1},v_{4,n-1})\nonumber\\
&&+N(K_{4}\times  P_{n-1}; v_{2,n-1},v_{3,n-1},v_{4,n-1})+N'(K_{4}\times  P_{n-1}; I_{n-1})\nonumber\\
&=& 2f_{4,n-1}^{1}+f_{4,n-1}^{2}+b_{n-1}^{4}+3f_{4,n-1}^{2}
+c_{n-1}^{1}+f_{4,n-1}^3+c_{n-1}^{4}+f_{4,n-1}^3+g_{n-1}\nonumber\\
&=&2f_{4,n-1}^{1}+4f_{4,n-1}^{2}+b_{n-1}^{4}+2f_{4,n-1}^{3}+c_{n-1}^{1}+c_{n-1}^{3}+g_{n-1},\nonumber
\end{eqnarray}
and the recurrence formulas for $b_{n}^{6}$ and $c_{n}^{3}$ can be derived in a similar manner.\par

Finally, by (\ref{equ3.16}) and Lemma~\ref{coro4} (2), (3),
\begin{eqnarray}
c_{n}^{1}&=&N'(K_{4}\times  P_{n}; v_{1,n},v_{2,n}, v_{3,n})\nonumber\\
&=&N'(K_{4}\times  P_{n-1}; v_{1,n-1})+N'(K_{4}\times  P_{n-1}; v_{2,n-1})+N'(K_{4}\times  P_{n-1}; v_{3,n-1})\nonumber\\
&&+N'(K_{4}\times  P_{n-1};v_{1,n-1},v_{2,n-1})+r_n+N'(K_{4}\times  P_{n-1}; v_{1,n-1},v_{4,n-1})\nonumber\\
&&+N'(K_{4}\times  P_{n-1}; v_{2,n-1},v_{3,n-1})
+N'(K_{4}\times  P_{n-1}; v_{2,n-1},v_{4,n-1})\nonumber\\
&&+N'(K_{4}\times  P_{n-1}; v_{3,n-1},v_{4,n-1})+N'(K_{4}\times  P_{n-1}; v_{1,n-1},v_{2,n-1},v_{3,n-1})\nonumber\\
&&+N'(K_{4}\times  P_{n-1}; v_{1,n-1},v_{2,n-1},v_{4,n-1})+t_n\nonumber\\
&&+
N'(K_{4}\times  P_{n-1}; v_{2,n-1},v_{3,n-1},v_{4,n-1})+N'(K_{4}\times  P_{n-1}; I_{n-1})\nonumber\\
&=&a_{n-1}^{1}+a_{n-1}^{2}+a_{n-1}^{3}+b_{n-1}^{1}+r_n+b_{n-1}^{6}+b_{n-1}^{2}+b_{n-1}^{5}+b_{n-1}^{3}
+c_{n-1}^{1}+c_{n-1}^{3}\nonumber\\
&&+t_n+c_{n-1}^{2}+g_{n-1}\nonumber\\
&=&a_{n-1}^{1}+2a_{n-1}^{2}+2b_{n-1}^{1}+b_{n-1}^{2}+b_{n-1}^{5}+b_{n-1}^{6}+r_{n}+2c_{n-1}^{1}+c_{n-1}^{3}+t_{n}+g_{n-1},\nonumber
\end{eqnarray}
and
\begin{eqnarray}
g_{n}&=&N'(K_{4}\times  P_{n}; I_{n})\nonumber\\
&=&a_{n-1}^{1}+a_{n-1}^{2}+a_{n-1}^{3}+a_{n-1}^{4}+b_{n-1}^{1}+b_{n-1}^{2}+b_{n-1}^{3}+2  r_{n}+k_{n}\nonumber\\
&&+c_{n-1}^{1}+c_{n-1}^{2}+2 t_{n}+g_{n-1}\nonumber\\
&=&2a_{n-1}^{1}+2a_{n-1}^{2}+2b_{n-1}^{1}+b_{n-1}^{2}+2  r_{n}+k_{n}+2c_{n-1}^{1}+2 t_{n}+g_{n-1}.\nonumber
\end{eqnarray}

The proof is complete.
\end{proof}

By Corollary~\ref{coro2} and Lemma~\ref{lem3}, we obtain an enumeration formula for $ N(P_{4}\times P_{n})$, which  is in alignment with Theorem~\ref{main}.

\begin{theorem}\label{main4}
For any positive integer $n$, the total number of   connected sets of $P_{4}\times  P_{n}$ is
$$N(P_{4}\times  P_{n})=\sum\limits_{k=1}^{n}(N_{4,k}-|\mathcal{C}_{4,k}'|)\cdot(n-k+1).$$
\end{theorem}
\noindent\textbf{Remark.} See some detailed calculation results below.\par
(1) When $n=1$, $|\mathcal{C}_{4,1}'|=5$,  $N_{4,1}^{*}=N(P_{4})=10=N_{4,1}-|\mathcal{C}_{4,1}'|$.\par
(2) For $n=2$, $n=3$ and $n=4$,\par
~~~~(i) $a_{2}^{1}=4$,  $a_{2}^{2}=2$, $b_{2}^{1}=5$, $b_{2}^{2}=4$,
$b_{2}^{4}=10$, $b_{2}^{6}=11$,  $c_{2}^{1}=3$, $c_{2}^{3}=12$, $g_{2}=0$, $|\mathcal{C}_{4,2}'|=87, N_{4,2}^{*}=88$;\par
~~~~(ii) $a_{3}^{1}=57$,  $a_{3}^{2}=39$, $b_{3}^{1}=76$, $b_{3}^{2}=68$,
$b_{3}^{4}=117$, $b_{3}^{6}=127$,  $c_{3}^{1}=64$, $c_{3}^{3}=134$, $g_{3}=40$, $|\mathcal{C}_{4,3}'|=1209, N_{4,3}^{*}=920$;\par
~~~~(iii) $a_{4}^{1}=749$,  $a_{4}^{2}=602$, $b_{4}^{1}=1037$, $b_{4}^{2}=968$,
$b_{4}^{4}=1479$, $b_{4}^{6}=1559$,  $c_{4}^{1}=999$, $c_{4}^{3}=1674$, $g_{4}=824$, $|\mathcal{C}_{4,4}'|=16431$, $N_{4,4}^{*}=9362$.\par

Using Mathematica, we calculate $N_{4,3}^{*}$ and $N_{4,4}^{*}$, and the results are consistent with those obtained from the  formulas.\par

\section{Concluding remarks }

In this paper, we have determined an enumeration formula for $N(K_{m}\times  P_{n}) $ by establishing the transfer matrix which
is actually obtained based on the recurrence relation for counting the connected sets of $K_{m}\times  P_{n} $. Due to
the structural characteristics of $K_{m}$, the transfer matrix is $m\times m$.
Further, we have provided a scheme in Lemma~\ref{lem1} to classify all the connected sets in $\mathcal C(K_{m}\times P_{n} )\setminus \mathcal C(P_{m}\times P_{n})$, which naturally deduces a sharp lower bound of this set and enables us to derive
enumeration formulas for $N(P_3\times P_n)$ and $N(P_4\times P_n)$. 
 Although the general formula for counting connected sets in any $(m\times n)$-grid graph  remains an open problem, the methodology presented in Lemmas~\ref{lem2} and~\ref{lem3} makes the counting for small $m$ possible.

\section*{Acknowledgements}

This work is supported by National Natural Science Foundation of China (No. 12171402).

\end{document}